\documentclass[12pt,leqno]{article}
 
 \usepackage{amsmath,amsthm,amssymb,amsfonts,hyperref}
\usepackage{mathtools}

\usepackage{pdfsync}
\usepackage[nobysame,initials]{amsrefs}
\usepackage{enumitem}
\usepackage{blindtext}
\usepackage{xspace}
 \usepackage{fancyhdr}
\usepackage{tikz-cd}
\usepackage{wrapfig}
\usepackage{framed}
\usepackage[english]{babel}
\usepackage[utf8]{inputenc}
\usepackage[T1]{fontenc}
\usepackage{comment}

\usepackage{esvect}
\usepackage{fourier-orns}
\usepackage{graphicx,type1cm,eso-pic,color}
\usepackage{fullpage}
\usepackage{float}
\usepackage{caption}
\usepackage{subcaption}
\usepackage{lipsum}

\theoremstyle{plain}
\newtheorem{thm}{Theorem}[section]
\newtheorem{lem}[thm]{Lemma}
\newtheorem{cor}[thm]{Corollary}
\newtheorem{prop}[thm]{Proposition}
\theoremstyle{definition}
\newtheorem{defi}[thm]{Definition}
\newtheorem{rem}[thm]{Remark}

\newcommand{\bproof}{\begin{proof}}
\newcommand{\eproof}{\end{proof}}

\renewcommand{\geq}{\geqslant}
\renewcommand{\leq}{\leqslant}

\DeclareMathOperator{\id}{{\mathrm id}}

\newcommand{\GL}{\mathrm{GL}}
\newcommand{\Sn}{{\mathbb S}_n}

\newcommand{\mcR}{\mathcal{R}}

\newcommand{\mcP}{\mathcal{P}}

\newcommand{\Cc}{{\mathbb{C}}}

\newcommand{\Nn}{{\mathbb{N}}}
\newcommand{\Zz}{{\mathbb{Z}}}

\newcommand{\beq}{\begin{displaymath}}
\newcommand{\eeq}{\end{displaymath}}
\newcommand{\be}{\begin{equation}}
\newcommand{\ee}{\end{equation}}

\newcommand{\conj}{\triangleright}
\newcommand{\grconj}{\blacktriangleright}

\numberwithin{equation}{section}

\usepackage{setspace}
\DeclareFontFamily{OT1}{rsfs}{}
\DeclareFontShape{OT1}{rsfs}{n}{it}{<-> rsfs10}{}
\DeclareMathAlphabet{\mathscr}{OT1}{rsfs}{n}{it}

\title{Twist equivalence and Nichols algebras over  Coxeter groups}
\author{Giovanna Carnovale, Gabriel Maret}

\begin{document}
 \maketitle

 \begin{abstract}
     Fomin-Kirillov algebras are quadratic approximations of Nichols algebras associated with the conjugacy class of transpositions in a symmetric group and a (rack) $2$-cocycle $q^+$ with values in $\{\pm1\}$.  Bazlov generalized their construction replacing the class of transpositions by the classes of reflections in an arbitrary finite Coxeter group. We prove that Bazolv's cocycle $q^+$ is twist-equivalent to the constant cocycle $q^-\equiv -1$, generalising a result of Vendramin. As a consequence, the Nichols algebras associated with the two different cocycles have the same Hilbert series and one is quadratic if and only if the other is quadratic. We further apply recent results of  Heckenberger, Meir and Vendramin and Andruskiewitsch, Heckenberger and Vendramin to complete the classification of the finite-dimensional Nichols algebras of Yetter-Drinfeld modules over the dihedral groups.
 \end{abstract}

\section{Introduction}
Racks and racks cocycles appeared in the study of knot theory, as they naturally provide a solution of the braid equation. It became apparent that they are key basic objects in the study of Nichols (shuffle) algebras, and the classification of pointed Hopf algebras \cite{AG}, since Nichols algebras are graded braided Hopf algebras that depend only on the datum of a braided vector space. Several tools have been developed in order to detect finiteness properties of Nichols algebras (such as finite-dimension, finite GK-dimension, finite presentation) that depend only on the rack and not on the cocycle, and properties that are preserved under suitable transformations of the cocycle,  see among others \cites{ACG, AFGV, AFGaV, AFGV1, AG, AHS, He2, HMV}. 

Relevant racks in the study of Nichols algebras can be  described as subsets $X$ of a group $G$ that are stable by conjugation, \cite{AG}*{Definition 1.3}. For such racks, \emph{twisting} a cocycle by a $2$-cocycle of $G$ with trivial coefficients corresponds, at the level of Nichols algebras, to performing a (dual) Drinfeld twist \cites{AFGV,MO}. As a consequence, Nichols algebras associated with the same rack and twist-equivalent cocycles have same Hilbert series,  same dimension, same GK-dimension. 

In the present paper we first focus on finitely generated Coxeter groups and a special family of racks and $2$-cocycles, that is of interest due to its relation with the cohomology of the flag variety. Here the rack is the class (or union of classes) of reflections in a  Coxeter group $W$, and the $2$-cocycle is a cocycle with values in $\{\pm1\}$. If $W$ is a Weyl group, the quadratic approximation of the corresponding Nichols algebra is a non-commutative algebra containing the cohomology algebra of the flag variety of the corresponding algebraic group. For $W={\mathbb S}_n$ this is the so-called Fomin-Kirillov algebra introduced in  \cite{FK} for creating an algebraic  and combinatorial framework to perform Schubert calculus. The construction was later generalized to the case of arbitrary Weyl groups in \cite{Bazlov}. For $W={\mathbb S}_3, {\mathbb S}_4$ and ${\mathbb S}_5$ the Nichols algebra in question is quadratic and finite-dimensional \cites{MS,GGI}. For $n>5$ a proof of infinite-dimensionality of the Fomin-Kirillov algebras appeared in \cite{Ba}, but it is not known whether the Nichols algebras are quadratic. This property, as well as dimension, is preserved by twists. 
  
A novel approach to the analysis of the dimension and of the degree of a minimal set of generating relations may come through geometry, using the category equivalence in \cite{KS}. Through this equivalence, Nichols algebras correspond to Intersection Cohomology  (IC) complexes on the infinite-dimensional space ${\rm Sym}({\mathbb C})$ of monic polynomials, with a so-called factorization datum. One may expect that the Nichols algebras associated with a conjugacy class  $C$ of a finite group $G$ and the constant cocycle equal to $-1$ might be easier to handle because the restriction of the corresponding IC-complexes to the dense open subset ${\rm Sym}_{\neq}({\mathbb C})$ of multiplicity-free polynomials is the push-forward through the covering map of the constant sheaf on the Hurwitz space with Galois group $G$ and local monodromy $C$ (\cite{KS}*{\S 3.3 A}, \cite{ETV}*{\S 2.4, 3.3}). These spaces are widely studied. For this reason, it becomes relevant to know whether the cocycle associated with the rack of reflections in Coxeter groups as in \cites{Bazlov,FK, MS} are twist-equivalent to the $-1$ constant cocycle. 

For $W={\mathbb S}_n$ this problem was answered in the affirmative in \cite{Vendramin}, and for the dihedral group an attempt was done in  \cite{Dold}. In the present paper we provide a unified approach to give a positive answer that in fact applies to all finitely generated Coxeter groups as long as the entries of its Coxeter matrix $A(W)$ are all finite.  We also show that the cocycles are cohomologous if and only if the entries  in $A(W)$ are all odd.  Hence, for dihedral groups of regular polygons with an odd number of edges, we recover the well-known fact that the Nichols algebras corresponding to these two cocycles are isomorphic, \cite{MS}*{\S 5}. 

We next combine the information we gained in order to reorganize and complement what is known about the dimension of Nichols algebras associated with reflections in Coxeter groups, showing that the key case to be studied is type $A_5$. Finally, we address  the case of Coxeter groups of rank $2$, that is, the dihedral groups $I_2(n)$ of order $2n$. This is of particular interest because all finite groups generated by $2$ involutions are dihedral. Finite-dimensional Nichols algebras over $I_2(n)$ for $4|n$ were classified in \cites{BC,FG}.  The case of $n=3$ is to be found in \cite{AHS} and $I_2(2)$ is not irreducible. Using the main results in \cites{AHV,HMV} and the work in  \cite{AF} we address the remaining open case, i.e., when $4$ does not divide $n$ and $n\neq 2,3$, completing the classification of finite-dimensional Nichols algebras over dihedral groups. In particular, we show that there is no finite-dimensional Nichols algebra over $I_2(n)$ when $n$ is odd and different from $3$. Hence, the only finite-dimensional complex pointed Hopf algebra with group of grouplikes isomorphic to $I_2(n)$ for such an $n$ is the group algebra. 

The paper is structured as follows: in \S \ref{SectionDefinition} we introduce the basic notions on Coxeter groups, racks, cocycles and Nichols algebras; in \S \ref{SectionGeneralizingVendramin} we define the cocycles $q^+$ and $q^-$ on the rack of reflections in an arbitrary Coxeter group and translate twist equivalence of $q^+$ and $q^-$ into existence of a section, generalizing the strategy in \cite{Vendramin}. The core argument is  in \S \ref{SectionGeneralProof} where we apply results in \cite{Stembridge} in order to show existence of the sought section (Theorem \ref{TheoremGabriel}). In  \S \ref{applications}  we collect and complement to what is known on Nichols algebras associated with the rack of reflections in finite Coxeter groups, we apply Theorem \ref{TheoremGabriel}  and the main results in \cites{AHV,HMV} in order to list the  Coxeter groups for which a conjugacy class  of reflections could possibly support a finite-dimensional Nichols algebra. Last section is devoted to the completion of the classification of finite-dimensional Nichols algebra over dihedral groups.

\section{Basic Notions}\label{SectionDefinition}

\subsection{Coxeter groups and Root systems}\label{sec:notation}
Let $(W,S)$ be a Coxeter system, with  preferred generating set $S=\{s_1,\,\ldots,\,s_l\}$, Coxeter graph $\Gamma(W)$, length function $\ell\colon W\to{\mathbb N}$ and Coxeter matrix $A(W)=(m_{ij})_{i,j=1,\,\ldots l}$, where $|s_is_j|=m_{ij}$. We further assume that $m_{ij}<\infty$ for any $i,j$. Following \cite{BjornerBrenti} we  fix a basis $\Delta=\{\alpha_1,\,\ldots,\,\alpha_l\}$ of ${\mathbb R}^l$ and define the  bilinear form $(,)$ on ${\mathbb R}^l$ by setting $(\alpha_i,\alpha_j) := -\text{cos}(\frac{\pi}{m_{ij}})$ for all $i,j \in \{1,\ldots,l\}$ and we identify $W$ with the image of the geometric representation obtained mapping $s_i\in S$ to the reflection with respect to $\alpha_i$. Then $\Phi=W\Delta$ is the set of roots and we denote by $\Phi^+$ and $\Phi^-$ the set of positive and negative roots, respectively. We also set \[T:=\{s_\alpha~|~\alpha\in\Phi^+\}=\{ws_\alpha w^{-1}~|~\alpha\in\Delta\}\] for the set of reflections in $W$ so $\Phi = \{\pm \alpha~|~s_\alpha\in T \} $. For $x\in T$  we indicate by $\alpha_x$ be unique root in $\Phi^+$ satisfying $x=s_{\alpha_x}$.

We will be mainly interested in $W$-conjugacy classes in $T$: there are as many as the number of connected components of the graph obtained from $\Gamma(W)$ by removing the edges with even label,  \cite{Bourbaki}*{Proposition 3, Chapitre IV, \S 1, Th\'eor\`eme 1, Chapitre VI, \S4}. 

We denote by $\Sigma = S^{(\Nn)}$ the  free group on $S$ whose elements are words in the alphabet $S$ and by $\pi\colon \Sigma\to W$  the projection. 
We denote by $a^{op}$ the  word \emph{opposite} to $a$, that is, the word obtained from $a$ by reversing  the order of its letters. Clearly, $\pi(a^{op})=\pi(a)^{-1}$.

For $x\in W$ the subset $\mcR(x)\subset \Sigma$ stands for the set of reduced expressions of $x$. An element $x\in W$ lies in $T$ if and only if it has a non-trivial palindromic reduced expression, that is, a non-empty reduced expression $r(x)$ of $x$ such that $r(x)=r(x)^{op}$, cf. \cite{Stembridge}*{Proposition 2.4}. 
For $x\in T$ we denote by $\mcP(x)$ the set of \emph{palindromic} reduced expressions of $x$. 

For the reader's convenience we recollect some very basic facts on length of conjugate reflections that will be needed in the sequel. 
\begin{lem}
\label{LemmaLinkArrowScalarRoots}
    Let $\beta\in\Phi^+$ and $\alpha \in \Delta$. Then
    \begin{enumerate}
        \item[(i)] $\ell(s_\alpha s_\beta s_\alpha) = \ell(s_\beta)+ 2   \iff 
        s_\beta(\alpha)\in\Phi^+\setminus\{\alpha\} \iff
        (\alpha ,\beta) < 0$ and  $ \beta \ne \alpha$.   
        \item[(ii)]
        $\ell(s_\alpha s_\beta s_\alpha) = \ell(s_\beta)- 2  \iff  s_\beta(\alpha)\in\Phi^-\setminus\{-\alpha\} \iff (\alpha ,\beta) > 0$ and $\beta\neq\alpha$.
        \item[(iii)] $\ell(s_\alpha s_\beta s_\alpha) = \ell(s_\beta) \iff s_\beta s_\alpha = s_\alpha  s_\beta \iff  (\alpha ,\beta) =0 $ or $ \beta=\alpha$. 
        \item[(iv)] If $\beta\in\Phi^+\setminus\Delta$, there is always $\alpha\in\Delta$ such that $\ell(s_\alpha s_\beta s_\alpha)=\ell(s_\beta)-2$.
    \end{enumerate}
\end{lem}

\bproof
We prove (i).   In virtue of  \cite{Bourbaki}*{Exercise V.4.8}, the condition $\ell(s_\alpha s_\beta s_\alpha) = \ell(s_\beta)+ 2$  is equivalent to $s_\beta(\alpha)\in\Phi^+$ and $s_\alpha s_\beta(\alpha)\in\Phi^+$, that it equivalent to $s_\beta(\alpha)\in\Phi^+\setminus\{\alpha\}$ by \cite{BjornerBrenti}*{Lemma 4.4.3}. Now, $s_\beta(\alpha)=\alpha-2(\alpha,\beta)\beta\in\Phi^+\setminus\{\alpha\}$ if and only if $(\alpha,\beta)<0$ and $\beta\neq\alpha$
Then, (ii) follows similarly and (iii) by exclusion. To prove (iv) assume for a contradiction that (ii) never occurs for $\beta$. Then $(\beta,\alpha)\leq 0$ for any $\alpha\in\Delta$ and thus also for any $\alpha\in\Phi^+$, whence $(\beta,\beta)\leq0$, a contradiction. 
\eproof

\subsection{Racks, cocycles and Nichols algebras}
\label{SectionRacks}

\begin{defi}
\label{DefinitionRack}
    A rack is a pair $(X,\conj)$ where $X$ is a non-empty set and $\conj: X\times X \to X $ is a function, such that the map $x \mapsto i \conj x $ is bijective for all $i\in X$, and $i\conj(j\conj k) = (i\conj j)\conj(i\conj k) $ for all $i,j,k \in X$.
\end{defi}

We will be mainly interested in racks of the form $X \subset G $, where $G$ is a group, $X$ is stable by conjugation and $x\conj y = xyx^{-1}$ for $x,y \in X$.

\begin{defi}
\label{DefinitionRackCocycle}\cite{AG}*{\S 2.2},\cite{CKS}*{\S 2}
    Let $X$ be a rack and let $A$ be an abelian group. A map $q: X \times X \longrightarrow A$ is a 2-cocycle (or just a cocycle) if  
  \begin{equation}\label{eq:cocycle-condition}q(x,y\conj z)q(y,z)=q(x\conj y,x\conj z)q(x,z) \end{equation}
 for all $x,y,z \in X$. We call $Z_R^2(X,A)$ the set of rack cocycles.  Two cocycles $q,q' \in Z_R^2(X,A)$ are said to be cohomologous if there exists $\gamma: X \to A$ such that: 
\begin{equation}\label{eq:cohomologous}q(x,y)=\gamma(x\conj y)^{-1}q'(x,y)\gamma(y),\mbox{ for all }x,y \in X.\end{equation} 
\end{defi} 

Our main motivation for studying racks and cocycles comes from Nichols algebras, whose construction we briefly recall. The base field here is $\mathbb C$.

Given a rack  $X$ and a cocycle $q$, the operator 
\begin{align*}
c_q\colon\Cc X\otimes \Cc X&\longrightarrow\Cc X\otimes \Cc X\\
x\otimes y&\mapsto q({x,y})x\conj y\otimes x
\end{align*}
satisfies the braid equation on $(\Cc X)^{\otimes 3}$:
\begin{equation}\label{eq:braid}(c_q\otimes \id)(\id\otimes c_q)(c_q\otimes \id)=(\id\otimes c_q)(c_q\otimes \id)(\id\otimes c_q).\end{equation}
We call any pair $(V,c)$ where $V$ is a vector space and $c\in \GL(V^{\otimes 2})$ satisfies \eqref{eq:braid} a braided vector space, and $c$ a braiding for $V$. 

We recall that, for $n\geq 2$, the $n$-th braid group $B_n$ is generated by $\sigma_i$, for  $i=1,\,\ldots,n-1$ with relations $\sigma_i\sigma_{i+1}\sigma_i=\sigma_{i+1}\sigma_i\sigma_{i+1}$ for $i\leq n-2$ and $\sigma_i\sigma_j=\sigma_j\sigma_i$ for $i,j=1,\,\ldots, n-1$ with $|i-j|>1$.
There is a unique section, called the Matsumoto section, $M\colon \Sn\to B_n$
satisfying $M(\sigma\tau)=M(\sigma)M(\tau)$ whenever $\ell(\sigma\tau)=\ell(\sigma)\ell(\tau)$.

Given a braided vector space $(V,c)$, for any $n\geq 2$, there is a representation
\begin{align*}\rho_n&\colon B_n\to \GL(V^{\otimes n})\\
\sigma_i&\mapsto \id^{\otimes(i-1)}\otimes c\otimes \id^{\otimes(n-i-1)}\end{align*}
that combined with the Matsumoto section, gives rise to the $n$-th quantum symmetrizer $\Omega_n:=\sum_{\sigma\in\Sn}\rho_n(M(\sigma))\in {\mathrm{End}}(V^{\otimes n})$.

The Nichols algebra associated with $(V,c)$  is the quotient of the tensor algebra ${\mathcal B}(V,c):=T(V)/\bigoplus_{n\geq0}{\mathrm{Ker}\,}\Omega_n$. If $(V,c)$ comes from a rack-cocycle pair $(X,c)$ we will also denote the Nichols algebra by ${\mathcal B}(X,q)$. It is usually very difficult to detect for which pairs $(X,q)$ the corresponding Nichols algebra is finite-dimensional, or has finite GK-dimension, or whether it is finitely presented. Big progresses have been obtained in order to detect properties of ${\mathcal B}(X,q)$ from the knowledge of $X$ only. However, there are cases on which dependence on the cocycle is apparent, \cite{HMV}*{Table 1} and some cocycles are easier to handle than others. 

\begin{rem}\label{rem:cohomologous}
\begin{enumerate}
    \item\label{item:isomorphic} If $X$ is a rack and $q,\,q'$ are cohomologous rack cocycles, then ${\mathcal B}(X,q)\simeq {\mathcal B}(X,q')$, cf. \cite{AG}*{Theorem 4.14}.
         \item \label{item:Nich-inclusion} Let $(V,c)$ be a braided vector space. If $V'$ is a subspace of $V$ such that $c(V'\otimes V')\subset V'\otimes V'$, then the inclusion $V'\subset V$ induces an injective algebra homomorphism ${\mathcal B}(V',c|_{V'\otimes V'})\to{\mathcal B}(V,c)$, \cite{AZ}*{Remark 1.1}. Hence, for any  pair $(X,q)$ and any inclusion $X'\to X$ such that $X'\conj X'\subset X'$ (we call $X'$ a subrack of $X$), we have an algebra inclusion
   ${\mathcal B}(X',q|_{X'\times X'})\subseteq {\mathcal B}(X,q)$. 
    \item\label{item:YD}Let $G$ be a group. A Yetter-Drinfeld module for $G$ is the datum of a $G$-graded representation $V=\bigoplus_{g\in G}V_g$ of $G$ such that $h V_g\subset V_{hgh^{-1}}$ for any $h,g\in G$. Setting $c_V(v\otimes w):=gw\otimes v$ for any $v\in V_g$ and $w\in V$ gives a braiding on $V$. 
    Yetter-Drinfeld modules decompose as direct sums of irreducible ones. The latter are, as representations, of the form $V={\rm Ind}_H^GU$, where $H$ is the centraliser of some $g\in G$ and $\eta\colon H\to \GL(U)$ is an irreducible representation of $H$. The $G$-grading on $V$ is determined by setting $V_g:=U$. If $\eta$ is $1$-dimensional, then there is a rack cocycle $q$ on the conjugacy class $X$ of $g$, with values in $\Cc^*$ such that $(V,c_V)=(\Cc X,q)$. For general $\eta$, one needs the more general notion of cocycle with values in $\GL(U)$, \cite{AFGV}*{Section 2.4}. If $V$ is a Yetter-Drinfeld module for $G$, we will write ${\mathcal B}(V)$ instead of ${\mathcal B}(V,c_V)$.  
   \item\label{item:autom}If $\tau\in{\rm Aut}(G)$ and $V=\bigoplus_{g\in G}V_g$ is a Yetter-Drinfeld module for a group $G$, then twisting $V$ by $\tau$ gives another Yetter-Drinfeld module structure $V^\tau$ by setting $g.v:=\tau(g)v$ and $(V^\tau)_g:=V_{\tau(g)}$ for any $g\in G$ and $v\in V$. The braidings $c_V$ and $c_{V^\tau}$ coincide, so ${\mathcal B}(V)\simeq{\mathcal B}(V^\tau)$. 
\item\label{item:tens-prod}  Let  $X$ be a rack admitting a subrack   decomposition $X=X_1\coprod X_2$ such that $X_i\conj X_j=X_j$ for $i,j\in \{1,2\}$, and let $q\in Z^2_R(X,\Cc)$. Assume in addition, that $c_q^2=\id$ on $\Cc X_i\otimes \Cc X_j$, for $i\neq j$, and that either all Nichols algebras involved are finite-dimensional or that $(X_i, q|_{X_i\times X_i})$ for $i=1,2$
are as in point \ref{item:YD}.  Then, there is an algebra isomorphism
${\mathcal B}(X_1,q|_{X_1\times X_1})\widetilde{\otimes} {\mathcal B}(X_2,q|_{X_2\times X_2})\simeq{\mathcal B}(X,q)$, where $\widetilde{\otimes}$ stands for the tensor product of algebras where the multiplication is twisted via $c_q$, \cite{GranaJ}*{Theorem 2.2}, \cite{HS}*{Proposition 1.10.12}.
\end{enumerate}
\end{rem}


\subsection{Group 2-cocycles and sections}
\label{SectionGroup}


We recall that if $G$ and $A$ are groups, with $A$ abelian, the group $Z^2(G,A)$ of $2$-cocycles of $G$ with values in $A$ consists of those maps $\phi:G \times G \to A$ satisfying 
 \begin{equation*}
\phi(xy,z) \phi(x,y) = \phi(x,yz)\phi(y,z),\quad\mbox{ for all } x,y,z \in G.
 \end{equation*}
 
If $E$ is a central extension of $G$, with conjugation action $\blacktriangleright$,  projection $\pi_G: E \to G$ and ${\rm Ker}(\pi_G)=A$, then for any section $\rho: G \to E $ of $\pi_G$, we will denote by 
$\phi_\rho$ the $2$-cocycle defined by 
$\phi_\rho(x,y):=\rho(xy)\rho(y)^{-1}\rho(x)^{-1}$ for $x,y\in G$. A direct calculation shows that  
\begin{equation}\label{eq:conj-cocycle}
\phi_\rho(x,y)(\rho(x)\blacktriangleright \rho(y)) =\phi_\rho(x\conj y,x)\rho(x\conj y),\quad \mbox{ for all }x,y\in G.\end{equation}


\begin{defi}
\label{DefinitionFlipEq}\cite{AFGV}*{\S 3.4}
    Let $G$ be a group, let $X \subset G$ be stable by conjugation and let $q,q' \in Z_R^2(X,A)$. We say that $q$ and $q'$ are twist equivalent if there exists $\phi \in Z^2(G,A)$ such that
    \begin{equation}\label{eq:twist-condition}q(x,y)=\phi(x,y)\phi(x\conj y,x)^{-1}q'(x,y), \quad\mbox{ for all }x,\,y\in X.\end{equation}
\end{defi}

\begin{rem}\label{rem:twist-equivalent}
If $X$ is a rack and $q,\,q'$ are twist-equivalent rack cocycles, then ${\mathcal B}(X,q)$ and ${\mathcal B}(X,q')$ are cocycle twist-equivalent, in particular, they have the same Hilbert series, \cite{MO}*{\S 2.7, 3.4},  \cite{AFGV}*{\S 3.4}.
\end{rem}

\begin{lem}
\label{PropositionGeneralTwistEq}
    Let $G$ be a group and let $X$ be a subset of $G$ stable by conjugation. Let $E$ be a central extension of $G$ with kernel $A$ and conjugation $\blacktriangleright$ and let $q_1,q_2 \in Z_R^2(X,A)$. If there is a section  $\rho: G \to E$ satisfying
    \begin{equation}\label{eq:twistEqGeneral}
        \rho(x) \blacktriangleright \rho(y) = q_1(x,y)^{-1}q_2(x,y) \rho(x\conj y) \ \text{for all} \ x,y \in X,
    \end{equation} then $q_1$ and $q_2$ are twist equivalent.
If, in addition, $E\simeq A\times G$ is a trivial extension, then $q_1$ and $q_2$ are cohomologous.
\end{lem}
\bproof The first statement follows from \eqref{eq:conj-cocycle} and \eqref{eq:twist-condition}. 
If $E\simeq A\times G$ is trivial and $\pi_A$ is the projection on $A$, then applying $\pi_A$ to \eqref{eq:twistEqGeneral} gives \eqref{eq:cohomologous} with $\gamma := \pi_A \circ \rho $.
\eproof

\subsection{Twist equivalence for the racks of reflections}
\label{SectionGeneralizingVendramin}


We now introduce 
 two specific cocycles $q^+$ and $q^-\in Z^2_R(T,\Cc^*)$, as restriction to $T\times T$ of the following functions on $W\times T$. Let $w\in W$ and $y\in T$. Recall that $\alpha_y$ is the positive root associated with $y$. 
We set:
\begin{equation}
\label{eq:Q+}
    q^+(w,y) =
    \begin{cases}
    1 \ $if$ \ w(\alpha_y) \in \Phi^+,\\
    -1 \ $if$ \ w(\alpha_y) \in \Phi^-,
\end{cases}\mbox{ and }\quad q^-(w,y) = \text{det}(w).
\end{equation}
%
%
 The function $q^+$ was introduced in \cite{Bazlov} for finite $W$, generalizing the $2$-cocycle for reflections in $\Sn$, \cites{MS, FK}. A straightforward calculation shows that $q^+$ and $q^-$ satisfy  
\begin{equation}
\label{eq:Cocycle2}
    q^\pm(w_1w_2,x) = q^\pm(w_1,w_2\conj x)q^\pm(w_2,x) \ \text{for all} \ w_1,w_2 \in W \text{ and } x \in T
\end{equation}
which implies \eqref{eq:cocycle-condition}, cf. \cite{MS}*{5.}.

The main result in \cite{Vendramin}*{Theorem 3.8} states that  if $ W = \Sn $, then $q^+$ and $q^-$ are cohomologous for $n=3$ and twist equivalent for $n\geq 4$. In the next Section we will prove the following generalization to (possibly infinite) Coxeter groups of this result.
\begin{thm}
\label{TheoremGabriel}
Assume that $A(W)$ has all its entries in ${\mathbb N}$. Then, the cocycles $q^+$  and $q^-$ on $T$ are twist equivalent. 
\end{thm}

\section{Proof of Theorem \ref{TheoremGabriel}}\label{SectionGeneralProof}

\subsection{Adapting Vendramin's construction}\label{SectionVendraminSection}
 
In \cite{Vendramin} the group cocycle inducing the twist equivalence is obtained by  costructing a suitable section from $\Sn$ to  the Schur covering group of $\Sn$. This approach cannot be slavishly reproduced to the case of an  arbitrary  Coxeter group $W$ because the Schur covering group of $W$ may be trivial or too large for our purposes, \cite{H}. We replace the Schur covering group with a suitable, possibly trivial, central extension of $W$.

Let $\widetilde{W}$ be the group generated by  $t_1,\ldots,t_l,\,z$ with relations 
\begin{equation}\label{eq:rel-lift}z^2 = (t_iz)^2 = 1, \quad  (t_it_j)^{m_{ij}} = z^{m_{ij}+1}, \mbox{ for all  }1 \le i,j \le l.\end{equation}

By construction, $z$ is central and the assignment $t_i\mapsto s_i$  for $i=1,\,\ldots,\,l$ and $z\mapsto 1$ defines 
a surjective homomorphism
$\pi_W : \widetilde{W} \longrightarrow W$ with kernel  
$\langle z \rangle  \cong \Zz/2\Zz$. 
\medbreak

We define the cocycles $q^+_z,q^-_z\in Z^2_R(T,\langle z\rangle)$ by replacing the $-1$ by $z$, so that,  for the non-trivial group isomorphism $\psi:\langle z\rangle \longrightarrow \{\pm 1\}$ we have $q^\pm = \psi\circ q^\pm_z$. It then suffices to show the twist equivalence of $q^+_z $ and $q^-_z $ by means of a suitable section $W\to \widetilde{W}$.
\begin{lem}
\label{PropositionVendraminSection2}
Let  $\rho:W \mapsto \widetilde{W}$ be a section of $\pi_W$. Assume that $\rho$ satisfies
    \begin{equation}\label{eq:vendra-property}
 \rho(s)\blacktriangleright \rho(y) =
    \begin{cases}
    \rho(s\conj y)z \ \ $if$ \ \ s \ne y \\
    \rho(s\conj y) \ \ $if$ \ \ s = y
\end{cases} \forall s \in S \mbox{ and } \forall y \in T. \end{equation}
Then \begin{equation}\label{eq:3} 
\rho(w) \grconj \rho(y) = q^+_z(w,y)q^-_z(w,y)^{-1}\rho(w\conj y),\quad \forall w\in W\mbox{  and }\forall y\in T.
\end{equation}
Hence, $q^+$ and $q^-$ are twist-equivalent.
\end{lem}
\bproof
We prove \eqref{eq:3}  by induction on the length of $w\in W$, the case $\ell(w)=1$ being \eqref{eq:vendra-property} in virtue of \cite{BjornerBrenti}*{Lemma 4.4.3}. 
Assume $\ell(w)>1$. Then $w = sw'$ for some $s\in S$ and $w' \in W$ such that $\ell(w') = \ell(w)-1$. Applying in sequence: the definition of $\phi_\rho$ from \eqref{eq:conj-cocycle}; centrality of $z$; the inductive hypothesis to $w'$ and $s$, and \eqref{eq:Cocycle2} we obtain: 
\begin{align*}
    \rho(w) \blacktriangleright \rho(y) &= \left(\phi_\rho(s,w')\rho(s)\rho(w')\right) \blacktriangleright \rho(y) \\
    &=  \rho(s) \blacktriangleright \left(\rho(w') \blacktriangleright \rho(y)\right) \\ 
    & = \rho(s) \blacktriangleright \left(q^+_z(w',y)q^-_z(w',y)^{-1}\rho(w'\conj y)\right) \\
    & =  q^+_z(w',y)q^-_z(w',y)^{-1} \rho(s) \blacktriangleright \rho(w'\conj y) \\ 
    & = q^+_z(w',y)q^-_z(w',y)^{-1}q^+_z(s,w'\conj y)q^-_z(s,w' \conj y)^{-1}\rho\left(s \conj (w'\conj y)\right)  \\
    & = q^+_z(sw',y)q^-_z(sw',y)^{-1}\rho(s w'\conj y).
\end{align*}
The last statement follows from Lemma \ref{PropositionGeneralTwistEq}. 
\eproof

\subsection{Towards the definition of a  section}\label{SectionGoodDefinition}


We build a section verifying \eqref{eq:vendra-property} using a  modified version of the conjugacy graph of $(W,S)$ in \cite{GeckPfeiffer}*{Section 3.2}. 

\begin{defi}
\label{DefinitionConjugacyGraph}
    The \textit{reflection conjugacy graph} $\tilde\Gamma(W)$ of $(W,S)$ is the labeled, directed  graph having the elements of $T$ as vertices, and labeled edges $x \xrightarrow{ \ s \ } y $ whenever $x,y\in T$ satisfy $y = s\conj x$ with $s\in S$ and $\ell(x)=\ell(y)+2$.
\end{defi}

To simplify notation we write $x\xrightarrow{ \ i \ } y$ instead of  $x \xrightarrow{ \ s_i \ } y$.  More generally, for a sequence  $s_{j_r},\ldots, s_{j_1}$ of elements in $S$, with $a=s_{j_r}\cdots s_{j_1}\in \Sigma$, we write $x \xrightarrow{ \ a \ } y $ if there is a path in $\widetilde{\Gamma}(W)$ from $x$ to $y$, with labels $j_r,\, j_{r-1},\ldots,\, j_1$, or, equivalently, if $x=\pi(a)\conj y$ with $\ell(x)=\ell(y)+2r$. 
Inductive application of Lemma \ref{LemmaLinkArrowScalarRoots} (iv)  guarantees that for any $x\in T$ there is always a path in $\widetilde{\Gamma}(W)$ starting in $x$ and ending  at some $s\in S$. 


\begin{rem}\label{rem:ReflectionReducedExpression}
By \cite{Stembridge}*{Proposition 2.4}, the set of reduced expressions of $x\in T$ consists of all words of the form $a_Lsa_R$ with $s\in S$ and $a_L,\,a_R \in \Sigma$, such that $x \xrightarrow{ \ a_L \ } s$ and $x \xrightarrow{ \ a_R^{op} \ } s$.  Hence, $\mcP(x)$ consists of all words of the form $a_Lsa_L^{op}$ where $s\in S$ and $x \xrightarrow{ \ a_L \ } s$. In other terms, $\mcP(x)$ is in bijection with the paths in $\widetilde{\Gamma}(W)$ starting from $x$ and ending in some $s\in S$. By construction all such paths have the same length. 
\end{rem}

We aim at defining a section of $\pi_W$ inductively. To do this we need to compare  different elements in ${\mathcal P}(x)$ for $x\in T$. First of all, we compare different elements in ${\mathcal R}(x)$. By \cite{BjornerBrenti}*{Theorem 3.3.1 (ii)} it suffices to look at reduced expressions that differ by application of one braid relation, a situation that was  considered by Stembridge. 


\begin{prop}(\cite{Stembridge}*{Lemma 2.5})
\label{PropositionMovesOfExpressions}
    Let $x\in T$ and let $a = a_Ls_ia_R$ and $b= b_Ls_jb_R \in {\mathcal R}(x)$, with $s_i,\,s_j\in S$, and $a_L, a_R, b_L, b_R\in \Sigma$. Assume that $a$ and $b$ differ by one braid relation. Then, one of the following three alternatives  holds:
    \begin{itemize}
    \item[\framebox{\bf Case $a \rightsquigarrow_L b$ }]\   $a_R=b_R$, $i=j$, $\pi(a_L)=\pi(b_L)$ and $a_L$ and $b_L$ differ by one braid relation;
        \item[\framebox{\bf  Case $a \rightsquigarrow_R b$ }] \  $a_L=b_L$, $i=j$, $\pi(a_R)=\pi(b_R)$ and $a_R$ and $b_R$ differ by one braid relation; 
        \item[\framebox{\bf Case $a \rightsquigarrow_C b$} ] \ $m_{ij}$ is odd, $i\neq j$, and there exist $a'_L,a'_R\in \Sigma$ such that
        \[a= a'_L(\underbrace{\ldots s_is_js_is_js_i\ldots}_{m_{ij}\ \rm terms})a'_R\mbox{ and }b= a'_L(\underbrace{\ldots s_js_is_js_is_j\ldots}_{m_{ij}\ \rm terms})a'_R.\]
  \end{itemize}  
\end{prop}\hfill$\Box$

We now focus on ${\mathcal P}(x)$.  For any $x\in T$ and any $a=a_Ls_ia_R\in \mcR(x)$ we consider the mirrored expression $\mu(a)=a_Ls_ia_L^{op}\in \Sigma$. By \cite{Dyer}*{(2.7)}, the assignment $\mu$ gives a well-defined function $\mu\colon  \mcR(x)\to \mcP(x)$, which restricts to the identity on $\mcP(x)$. 

\begin{prop}
\label{CorSymmetricMoves}
    Let $x \in T$, and let $a = a_Ls_ia_L^{op}$ and $b= b_Ls_jb_L^{op} \in \mcP(x)$. Then there exists a sequence 
 $a_d:=a_{d,L} s_{i_d} a_{d,L}^{op}\in\mcP(x)$, for $d=0,\,\ldots,\, r,$ with $a_0=a$ and $a_r=b$ and each $s_{i_d}\in S$ such that for every $d=0,\,\ldots,\,r-1$ either 
$a_{d+1,L}$ differs from $a_{d,L}$ by one braid move, or else $a_{d}\rightsquigarrow_C a_{d+1}$.     
\end{prop}
\bproof
By Proposition \ref{PropositionMovesOfExpressions} there is a sequence $b_c:=b_{c,L} s_{i_c} b_{c,R}\in\mcR(x)$, for $c=0,\,\ldots,\, q,$ with $b_0=a$ and $b_q=b$ and each $s_{i_c}\in S$ such that  $b_c \rightsquigarrow_{D_c} b_{c+1}$ for some $D_c\in \{L,R,C \}$ and for all $c\in \{0,\ldots,q-1\}$.  Applying the mirroring function $\mu$ to all terms we obtain a sequence 
$\mu(b_0)= \mu(a)= a,\,\mu(b_1),\ldots,\mu(b_q)=b \in \mcP(x)$. Now, if $D_c=R$, then $b_c=b_{c,L} s_{i_c} b_{c,R}\rightsquigarrow_{R}b_{c,L} s_{i_c} b_{c+1,R}=b_{c+1}$ and so $\mu(b_c)=\mu(b_{c+1})$.  If $D_c=L$, then $b_c=b_{c,L} s_{i_c} b_{c,R}\rightsquigarrow_{R}b_{c+1,L} s_{i_c} b_{c,R}=b_{c+1}$ with $b_{c,L}$ and $b_{c+1,L}$ differing by one braid move. Hence, $\mu(b_c)=b_{c,L} s_{i_c} b_{c,L}^{op}$ and $\mu(b_{c+1})=b_{c+1,L} s_{i_c} b_{c+1,L}^{op}$ with $b_{c,L}$ and $b_{c+1,L}$ differing by one braid move. Finally, if $D_c=C$, then 
\begin{equation*}
b_c=b'_{c,L}(\ldots s_{i_c}s_{i_{c+1}}s_{i_c}s_{i_{c+1}}s_{i_c}\ldots)b'_{c,R}\rightsquigarrow_{C} b'_{c,L}(\ldots s_{i_{c+1}}s_{i_c}s_{i_{c+1}}s_{i_c}s_{i_{c+1}}\ldots)b'_{c,R} =b_{c+1}   
\end{equation*} for some $b'_{c,L}, b'_{c,R}\in \Sigma$.
Thus, $\mu(b_c)=b'_{c,L}(\ldots s_{i_{c+1}}s_{i_c}s_{i_{c+1}}\ldots)b_{c,L}^{'op}$ and $\mu(b_{c+1})=b'_{c,L}(\ldots s_{i_c}s_{i_{c+1}}s_{i_c}\ldots)b_{c,L}^{'op}$, so $\mu(b_c)\rightsquigarrow_{C}\mu(b_{c+1})$.
Removing redundancy, we obtain the desired sequence in $\mcP(x)$. 
 \eproof

We are now in a position to inductively define a section of $\pi_W$ with good properties. 
 
\begin{lem}
\label{LememSectionGood}
Let  $\rho_0: W \longrightarrow \widetilde{W}$ be any set theoretic section of $\pi_W$. 
The assignment 
\begin{align}\label{eq:DefinitionSectionWithGraph}
\rho&\colon W \longrightarrow \widetilde{W},&&
x \mapsto 
\begin{cases}
    \rho_0(x)\ $ if $ x\notin T,\\
    t_i $ if $ x= s_i $ for some $ 1 \le i \le l, \\
    t_i \blacktriangleright \rho(y)z $ if $ x \xrightarrow{ \ i \ } y$ for some $ 1 \le i \le l,
\end{cases}
\end{align}
determines a well-defined section of $\pi_W$. 
\end{lem}
\bproof
For $x\in T$, let 
\[x=x_0\xrightarrow{ \ j_1 \ }x_1 \xrightarrow{ \ j_2 \ }\cdots x_{r-1}\xrightarrow{ \ j_r \ }s_i\]
be a path in $\widetilde{\Gamma}(W)$ ending at some $s_i\in S$. Applying \eqref{eq:DefinitionSectionWithGraph}, using that $z\in Z(\widetilde{W})$, gives
\begin{equation}\label{eq:computation-rho}\rho(x)=t_{i_1}\blacktriangleright(t_{i_2}\blacktriangleright(\cdots \;t_{i_r}\blacktriangleright t_i)\cdots)z^r=t_{i_1}t_{i_2}\cdots t_{i_r}t_it_{i_r}\cdots t_{i_2}t_{i_1}z^r,\end{equation}
and $\rho$ is well-defined if the term on the right hand side is independent from the chosen path, or, equivalently, independent from the corresponding reduced expression $a= s_{i_1}\cdots s_{i_r}s_js_{i_r}\cdots s_{i_1}$ in $\mcP(x)$. By Proposition \ref{CorSymmetricMoves}, it is enough to verify that $\rho(x)$ does not change if $a$ is modified by applying either one braid move to $a_L=s_{i_1}\cdots s_{i_r}$ (and consequently to  $a_L^{op}$)  or 
 by applying the move $\rightsquigarrow_C$  to $a$. 
In the first scenario, a subword of the form $\underbrace{s_as_b\cdots }_{m_{ab}}$ in $a_L$ is replaced by $\underbrace{s_bs_a\cdots }_{m_{ab}}$. In this case, a subword of the form $\underbrace{t_at_b\cdots }_{m_{ab}}$ occurring before $t_i$ in \eqref{eq:computation-rho} is replaced by $\underbrace{t_bt_a\cdots }_{m_{ab}}$, and symmetrically a subword of the form $\underbrace{t_bt_a\cdots }_{m_{ab}}$ occurring after $t_i$  in \eqref{eq:computation-rho} is replaced by $\underbrace{t_at_b\cdots }_{m_{ab}}$. Since $\underbrace{t_at_b\cdots }_{m_{ab}}=z^{m_{ab}+1}\underbrace{t_bt_a\cdots }_{m_{ab}}$ and $z$ is central, the value of $\rho(x)$ does not change in this case. In the second scenario, the term $a=a_Ls_ia_R= a'_L(\underbrace{\ldots s_js_is_j\ldots}_{m_{ij}})a'_R$ is replaced by $a'_L(\underbrace{\ldots s_is_js_i\ldots}_{m_{ij}})a'_R$ and $m_{ij}$ is odd. In this case, the central term $\underbrace{\ldots t_jt_it_j\ldots}_{m_{ij}}$ in \eqref{eq:computation-rho} is replaced by $\underbrace{\ldots t_it_jt_i\ldots}_{m_{ij}}$, while the rest is unmodified. As 
$\underbrace{\ldots t_jt_it_j\ldots}_{m_{ij}}=z^0\underbrace{\ldots t_it_jt_i\ldots}_{m_{ij}}$ in $\widetilde{W}$, the value of $\rho(x)$ is unaltered, so $\rho$ is well-defined. Applying $\pi_W$ to \eqref{eq:computation-rho} shows that it is a section.
\eproof

\subsection{Good properties of the section} \label{SectionVerifies(2)}

Next step is to prove that the section $\rho$ as in \eqref{eq:DefinitionSectionWithGraph} satisfies Vendramin's condition. We fix some further notation. 

\begin{defi}
\label{DefinitionChebychev}
The sequence $(U_n(X))_{n\geq 0}$ of Chebychev polynomials of the second kind is the sequence of polynomials in ${\mathbb Z}[X]$ defined recursively by:
\begin{align}
&U_0(X) = 1,&&U_1(X) = 2X, && U_{n+1}(X) = 2X U_{n}(X) - U_{n-1}(X).
\end{align}
\end{defi}
We recall the well-known formula
\begin{equation}\label{eq:Cheby}\text{sin}(\theta) U_n(\cos\theta) = \text{sin}((n+1)\theta).\end{equation}

The following lemma is key for proving inductively the good properties of $\rho$.

\begin{lem}\label{LemmaChebychevIsAwesome}
Let $\beta\in\Phi^+\setminus\Delta$ and assume there are $\alpha_i,\alpha_j\in\Delta$ satisfying \\ $\delta:=(\alpha_j,\beta)>0$ and $(\alpha_i,\beta)=0$.
For $p\geq0$, let
\begin{align*}
&\alpha_{(p)}=\begin{cases}\alpha_i&\textrm{ if }p\textrm{ is even,}\\
\alpha_j&\textrm{ if }p\textrm{ is odd,}
\end{cases}&s_{(p)}:=s_{\alpha_{(p)}},
&&\beta_0:=\beta, &&\beta_{p+1}:=s_{(p+1)}\beta_{(p)}.
\end{align*}
Then, 
\begin{equation}\label{eq:positive}
(\beta_p,\alpha_{(p+1)})>0, \textrm{ for }p=0,\,\ldots,\,m_{ij}-2,\quad \textrm{ and }\quad 
(\beta_{m_{ij}-1},\alpha_{(m_{ij})})=0\end{equation}
and 
either 
\begin{equation}\label{eq:length}
\ell(s_{\beta_{m_{ij}-1}})=\ell(s_\beta)-2m_{ij}+2, \textrm{ and }s_{(m_{ij})}s_{\beta_{m_{ij}-1}}=s_{\beta_{m_{ij}-1}}s_{(m_{ij})} \textrm{ and }\alpha_{(m_{ij})}\neq\beta_{m_{ij}-1}
\end{equation}
or else
\begin{equation}\label{eq:second-option}
m_{ij}\textrm{ is even, }
\beta_{m_{ij}/2-1}=\alpha_{(m_{ij}/2)}\in\Delta,\textrm{ and  }\ell(s_{\beta})=2m_{ij}-1.\end{equation}
\end{lem}
\bproof
For $p\geq 0$ we set 
\begin{align*}u_p:=(\beta_p,\alpha_{(p+1)})/\delta=(\beta_p,\alpha_{(p-1)})/\delta, &&\gamma:=-(\alpha_i,\alpha_j)=\cos(\pi/m_{ij}).\end{align*} Then $u_0=1$. 
For $p\geq 1$ we compute
\begin{align*}
    u_1&=(s_j\beta,\alpha_i)/\delta=(\beta-2(\beta,\alpha_j)\alpha_j,\alpha_i)/\delta=0+2\gamma.\\
u_{p+1}&=(s_{(p+1)}\beta_p,\alpha_{(p)})/\delta=(\beta_p-2(\beta_p,\alpha_{(p+1)})\alpha_{(p+1)},\alpha_{(p)})/\delta\\
&=(\beta_p,\alpha_{(p)})/\delta-2(\beta_p,\alpha_{(p-1)})(\alpha_{(p+1)},\alpha_{(p)})/\delta=(s_{(p)}\beta_{p-1},\alpha_{(p)})/\delta+2\gamma u_{p}\\
&=(\beta_{p-1},s_{(p)}\alpha_{(p)})/\delta+2\gamma u_{p}=-u_{p-1}+2\gamma u_{p}.
\end{align*}
Therefore $(\beta_p,\alpha_{(p+1)})/\delta=U_p(\gamma)$, the $p$-th Chebychev's polynomial evaluated at $\gamma=\cos(\pi/m_{ij})$. Then, \eqref{eq:Cheby} gives \eqref{eq:positive} and Lemma \ref{LemmaLinkArrowScalarRoots} (i) and (ii) applied to $\beta_{(p)}$ and $\alpha_{(p+1)}$ imply that either
\begin{equation}\label{eq:length-geq}
\ell(s_{\beta_{p+1}})< \ell(s_{\beta_{p}})\textrm{ for }p=0,\,\ldots,\,m_{ij}-2, 
\end{equation}
giving \eqref{eq:length},  or else  $\ell(s_{\beta_{q+1}})=\ell(s_{\beta_{q}})$  for some $q\in\{0,\,\ldots,\,m_{ij}-2\}$.  
In this case, Lemma \ref{LemmaLinkArrowScalarRoots} (iii) gives 
${\beta_q}=\alpha_{(q+1)}$,  whence
\begin{equation}\label{eq:red-expr-sb}s_\beta=s_{(1)}s_{(2)}\cdots s_{(q)}s_{(q+1)}s_{(q)}\cdots s_{(1)}.\end{equation}

Applying the equality $s_is_\beta s_i=s_\beta$ yields 
\[s_is_{(1)}s_{(2)}\cdots s_{(q)}s_{(q+1)}s_{(q)}\cdots s_{(1)}=s_{(1)}s_{(2)}\cdots s_{(q)}s_{(q+1)}s_{(q)}\cdots s_{(1)}s_i\]
where both sides are products of $2q+2$ simple reflections. Bearing in mind that $s_{(1)}=s_j$, this gives $(s_is_j)^{2q+2}=1$ and so $m_{ij}$ divides $2(q+1)$. By construction $m_{ij}> q+1$, hence $m_{ij}=2(q+1)$ is even and $\beta_{m_{ij}/2-1}=\alpha_{(m_{ij}/2)}$.
The claim on the length follows because $\ell(s_{\beta_{p+1}})=\ell(s_{\beta_{p}})-2$ for $p\in\{0,\,\ldots,\,q-1\}$. 
\eproof

\begin{lem}
\label{PropositionRhoSatisfies(2)}
   The section  $\rho $ defined by  \eqref{eq:DefinitionSectionWithGraph}
   satisfies condition \eqref{eq:vendra-property}. 
\end{lem}
\bproof
Let $y\in T$ and $s_i \in S$. Then we fall in one of the following situations: $y=s_i$,  or $\ell(s_i\conj y)=\ell(y)\pm2$, or $(\alpha_y,\alpha_i)=0$,  that we analyse separately. 
\begin{itemize}
    \item If $y=s_i$, then $\rho(s_i) \grconj \rho(y) = t_i\grconj t_i=t_i=\rho(y)$. 
    \item If $\ell(s_i\conj y)=\ell(y)+2$, then by construction
     $\rho(s_i\conj y) = t_i \grconj \rho(y) z= \rho(s_i) \grconj \rho(y)z $. Multiplying both sides by $z$  gives \eqref{eq:vendra-property}. 
    \item  If $\ell(s_i\conj y)=\ell(y)-2$, then we  invoke the previous case applied to $x=s_i\conj y$ and use that $t_i$ and $z$ are involutions and that $z$ is central. 
    \item If $(\alpha_y,\alpha_i)=0$ then $y \ne s_i$ and $s_i\conj y = y$. We proceed by induction on the length of $y$. If $\ell(y)=1$, then 
    $y=s_j$ for some $j\in\{1,\ldots,\,l\}$, with $m_{ij}=2$ and 
    $\rho(s_i\conj y)=\rho(y)=t_j=zt_i\grconj t_j=\rho(s_i)\grconj \rho(s_j)z$, confirming \eqref{eq:vendra-property} in this case. Assume now that $\ell(y)>1$. By Lemma \ref{LemmaLinkArrowScalarRoots}, there always is a $j\in\{1,\ldots,\,l\}$ satisfying $\ell(s_jy s_j)=\ell(y)-2$. Setting $\beta=\alpha_y$,   Lemma \ref{LemmaChebychevIsAwesome}, from which we retain notation,  gives either \eqref{eq:length} or else \eqref{eq:second-option}.

If \eqref{eq:second-option} holds, then   $y=s_{(1)}\conj(s_{(2)}\conj\dots(s_{(m_{ij}/2-1)}\conj s_{(m_{ij}/2)}))$ and $\ell(y)=2m_{ij}-1$, so  
\[\rho(s_i\conj y)=\rho(y)=(t_jt_i)^{m_{ij}/2-1}t_jz^{m_{ij}/2-1}.\]
As $(t_it_j)^{m_{ij}} = z$, we have $(t_it_j)^{m_{ij}/2} = z(t_jt_i)^{m_{ij}/2}$. Then,
\begin{align*}
    \rho(s_i)\grconj\rho(y)z&=t_i(t_jt_i)^{m_{ij}/2-1}t_jt_iz^{m_{ij}/2}=t_i(t_jt_i)^{m_{ij}/2}z^{m_{ij}/2}\\
    &=t_i(t_it_j)^{m_{ij}/2}z^{m_{ij}/2+1}=(t_jt_i)^{m_{ij}/2-1}t_jz^{m_{ij}/2-1}=\rho(s_i\conj y).
\end{align*}

If instead, \eqref{eq:length} holds, then 
$y=s_{(1)}\conj (s_{(2)}\conj\dots(s_{(m_{ij}-1)}\conj s_{\beta_{m_{ij}-1}})\cdots)$ with $\ell(s_{\beta_{m_{ij}-1}})=\ell(y)-2m_{ij}+2<\ell(y)$ and 
$s_{(m_{ij})}\conj s_{\beta_{m_{ij}-1}}=s_{\beta_{m_{ij}-1}}$, so by induction $t_{(m_{ij})}\grconj \rho(s_{\beta_{m_{ij}-1}})=\rho(s_{\beta_{m_{ij}-1}})z$. 
By definition of $\rho$ we have 
\[\rho(s_i\conj y)=\rho(y)=\underbrace{t_jt_i\cdots }_{m_{ij}-1\ \rm terms}\rho(s_{\beta_{m_{ij}-1}})\underbrace{\cdots t_it_j}_{m_{ij}-1\ \rm terms}z^{m_{ij}-1}.\]
Therefore 
\begin{align*}
   \rho(s_i)\grconj\rho(y)z&=\underbrace{t_it_jt_i\cdots }_{m_{ij}\ \rm terms}\rho(s_{\beta_{m_{ij}-1}})\underbrace{\cdots t_it_jt_i}_{m_{ij}\ \rm terms}z^{m_{ij}}\\
   &=\underbrace{t_jt_it_j\cdots }_{m_{ij}\ \rm terms}\rho(s_{\beta_{m_{ij}-1}})\underbrace{\cdots t_jt_it_j}_{m_{ij}\ \rm terms}z^{m_{ij}+2m_{ij}+2} 
\end{align*}
where the last factor in $t_it_j\cdots$ is now $\rho(s_{(m_{ij})-2})=\rho(s_{(m_{ij})})$. By induction, 
\[\rho(s_i)\grconj\rho(y)z=\underbrace{t_jt_i\cdots }_{m_{ij}-1\ \rm terms}\rho(s_{\beta_{m_{ij}-1}})\underbrace{\cdots t_it_j}_{m_{ij}-1\ \rm terms}z^{m_{ij}+1}=\rho(s_i\conj y)\]
concluding the proof.
\end{itemize}
\eproof

We end this Section characterising when $q^+$ and $q^-$ are cohomologous. For the case of ${\mathbb S}_n$, see \cite{Vendramin}*{Remark 2.2}.
\begin{thm}\label{thm:cohomologous} Assume that $A(W)$ has all its entries in ${\mathbb N}$. Then, the following are equivalent:
\begin{enumerate}
    \item All the coefficients in $A(W)$ are odd.
    \item The group $\widetilde{W}$ is the trivial extension $\widetilde{W}= W\times \langle z\rangle$.
    \item The cocycles $q^+$ and $q^-$ are cohomologous.
    \end{enumerate}\end{thm}
    \bproof The implication $1.\Rightarrow 2.$ is immediate from the definition of $\widetilde{W}$, and $2.\Rightarrow 3.$ follows from  from Lemmata \ref{PropositionGeneralTwistEq}
 \ref{PropositionVendraminSection2}, and \ref{PropositionRhoSatisfies(2)}. We prove that $3. \Rightarrow 1.$ Suppose for a contradiction that $m:=m_{ij}$ is even for some $i,j\in\{1,\,\ldots,\,l\}$, and that there exists $\gamma\colon W \longrightarrow \Cc^*$ such that 
$q^-(x,y)=\gamma(x\conj y)^{-1}q^+(x,y)\gamma(y)$ for all $x,y \in T$. Set $s=s_i$, $s'=s_j$ and let $\sigma := (s's)^{m/2-1}s'$, so that $\sigma\in T$ and $s\sigma s = \sigma$. Then 
$\gamma(s\conj \sigma) = \gamma(\sigma)$, thus $-1=q^-(s,\sigma)=q^+(s,\sigma)$, contradicting \eqref{eq:Q+}. 
\eproof

\begin{rem}If all coefficients in $A(W)$ are odd, then the underlying graph of  $\Gamma(W)$ is complete. Hence, the corresponding group $W$ is finite if and only if $l\leq 2$. If $l=1$, then $T=\{s_\alpha\}$ and $q^+=q^-$ so the statement is trivial. If $l=2$, then $W$ is of type $I_2(2m+1)$, for $m\geq 1$, i.e., it is the  dihedral group of a regular $(2m+1)$-gon, \cite{Bourbaki}*{Chapitre VI, \S 4, Th\'eor\`eme 1}. In this case, the cohomology of $q^+$ and $q^-$ is proved in \cite{MS}*{Example 5.4 (a)}. Notice that, even though $\widetilde{W}$ is trivial, by \cite{H}, the Schur multiplier of $W$ is elementary abelian of order $2^{(l-1)(l-2)/2}$, whence non-trivial whenever $W$ is infinite. 
\end{rem}

\section{Applications to Nichols algebras}\label{applications}

In this section $W$ is finite and the base field is $\mathbb C$. The Nichols algebra associated with $(T, q^{+})$ is of particular interest because it contains the coinvariant algebra of $W$, \cites{FK,Bazlov}. If $W$ is crystallographic, the latter is isomorphic to the cohomology of the flag variety of the algebraic group with associated Weyl group $W$. It is in general an open question whether ${\mathcal B}(T, q^{+})$ is finite-dimensional, quadratic, or finitely presented, \cites{GGI,Bazlov,FK,MO}. The main result in \cite{Ba} states that the quadratic cover of ${\mathcal B}(T,q^+)$ is infinite-dimensional for $W={\mathbb S}_n$ and $n\geq 6$. 

Combining Theorems \ref{TheoremGabriel} and Remark  \ref{rem:twist-equivalent} readily gives the following. 

\begin{cor}\label{cor:gabriel}
If $W$ is an arbitrary finite Coxeter group, then ${\mathcal B}(T, q^{+})$ and ${\mathcal B}(T, q^{-})$ have the same Hilbert series. This also holds for their quadratic approximations. In addition, 
${\mathcal B}(T, q^{+})$ is quadratic if and only if ${\mathcal B}(T, q^{-})$ is quadratic. \end{cor}

It also follows from Theorem \ref{thm:cohomologous} and Remark \ref{rem:cohomologous} point \ref{item:isomorphic} that if $W$ is of type $I_2(2m+1)$, for $m\geq 1$, that is, the dihedral group of a regular $(2m+1)$-gon, then ${\mathcal B}(T, q^{+})\simeq {\mathcal B}(T, q^{-})$, a result that was already present in \cite{MS}*{Example 5.4 (b)}.  By \cite{MS}*{Remark 5.2 part 2)} the braided vector spaces attached to $(T, q^{\pm})$ are constructed as in Remark \ref{rem:cohomologous} point \ref{item:YD}, so results in this setting apply.
We summarize here below some facts on the Nichols algebras ${\mathcal B}(T, q^{\pm})$. Some of these results are known, we focus on giving a uniform point of view.

\begin{rem}\label{rem:deco}
\begin{enumerate}
    \item If $(W,S)$ is not irreducible, let
$(W_1,S_1),\,\ldots ,(W_r,S_r)$ be its irreducible factors. Setting $T_i:=W\conj S_i=W_i\conj S_i$ for $i=1,\,\ldots, r$ we have a rack decomposition $T=\coprod_{i=1}^rT_i$ where $t\conj t'=t'$ and $q^+(t,t')=q^+(t',t)=1$ if $t\in T_i$ and $t'\in T_j$ with $i\neq j$. 
Hence $c_{q^+}^2$ and $c_{q^-}^2$ act as the identity on $\Cc T_i\otimes \Cc T_j$ whenever $i\neq j$. Remark \ref{rem:cohomologous} point \ref{item:tens-prod} then gives ${\mathcal B}(T,q^+)\simeq \bigotimes_{i=1}^r{\mathcal B}(T_i,q^+_i)$ and ${\mathcal B}(T,q^-)\simeq {\bigotimes}_{i=1}^r{\mathcal B}(T_i,q^-_i)$, where $q_i^{\pm}$ stands for the restriction of $q^{\pm}$ to $T_i\times T_i$. Therefore, it is enough to study the case of irreducible Coxeter groups.
\item\label{item:inclusions} Coxeter graph inclusions imply inclusions of the corresponding racks of reflections, hence a braided vector space inclusion for the cocycle $q^-$, and therefore a Nichols algebra inclusion by Remark \ref{rem:cohomologous} point \ref{item:Nich-inclusion}. 
\item\label{item:deco} Let $(W,S)$ be irreducible. If $\Gamma(W)$ has no even labeled edges, then the reflections form a single conjugacy class. Otherwise, $W$ is of type $B_l$ for $l\geq 3$, $F_4$, or $I_2(2m)$ for $m\geq 2$ and $T$ is the union of two classes,  represented by any $s$ and $s'\in S$  that are joined in $\Gamma(W)$ by an even labeled edge. Setting $T_1:=W\conj s$ and $T_2:=W\conj s'$ we have a rack decomposition $T=T_1\coprod T_2$ with $T_i\conj T_j=T_j$ for $i,j\in\{1,2\}$.
Observe that 
$c^2_{q^{\pm}}(s\otimes s')=\pm (s\conj s')\conj s\otimes s\conj s'\neq s\otimes s'$, so $c_{q^{\pm}}^2\neq\id$ on $\Cc T_1\otimes \Cc T_2$. If $\min(|T_1|,|T_2|)>2$ or $\max (|T_1|,|T_2|)>4$, then $\dim{\mathcal B}(T,q)=\infty$ for \emph{any} cocycle $q$,  by \cite{ACG}*{Theorem 2.9} applied to $Y=T$.
By construction, for $i=1,2$, we have twist equivalence of the restrictions to $T_i$ of $q^{+}$ and $q^{-}$.  
\item If $W$ is of type $B_l$ with $l\geq 3$, up to renumbering, $T_1$ is isomorphic to the rack $T_{D_l}$ of reflections for $W$ of type $D_l$, so $|T_1|=l^2-l$ and $T_2$ is abelian, that is, it has trivial action, and $|T_2|=l$. By \cite{ACG}*{Theorem 2.9},  $\dim{\mathcal B}(T,q)=\infty$ for any cocycle $q$. On the other hand,  ${\mathcal B}(T_2, q^-)=\bigwedge \mathbb C^l$ and ${\mathcal B}(T_2, q^+)=(\bigwedge \mathbb C)^{\otimes l}$ hence $\dim{\mathcal B}(T_2,q^{\pm})=2^l$. 
\item\label{item:dihedral} If $W$ is of type $F_4$ or  $I_2(2m)$ for $m\geq 2$, then the classes $T_1$ and $T_2$ are interchanged by the automorphism of $W$ coming from the symmetry of $\Gamma(W)$, so $T_1\simeq T_2$ as racks.  Using the descriptions of roots in \cite{Bourbaki}*{Planches IV, VIII}, one sees that in type $F_4$ the racks $T_1$ and $T_2$ are isomorphic to $T_{D_4}$, so $|T_1|=|T_2|>4$ whence $\dim{\mathcal B}(T,q)=\infty$ for \emph{any} cocycle $q$.
\item If $W$ is of type $I_2(2m)$ for $m\geq 2$, the racks $T_1$ and $T_2$ are isomorphic to the rack $T_{I_2(m)}$ of reflections of type $I_2(m)$.  In addition, it was shown in  \cite{AFGaV}*{Lemma 2.1} using \cite{AFGV1}*{Theorem 3.6} that if $m>2$, then $\dim{\mathcal B}(T_{I_2(2m)},q)=\infty$ for \emph{any} cocycle $q$. Several considerations concerning the rack $T$ for dihedral groups are present in \cite{MS}*{Section 5,6}. In particular, \cite{MS}*{Example 6.5} shows that $\dim{\mathcal B}(T_{I_2(4)},q^{\pm})=64$. The Nichols algebras and/or their quadratic covers for the rack of reflections in $I_2(2p)$ for $p$ an odd prime had been studied in \cite{AG1}*{Example 3.3.5}, \cite{AG}, \cite{Michel}. 
\end{enumerate}
\end{rem}
In the Tables below we summarize what is currently known about the Nichols algebras of the pair $(T,q^+)$ for $W$ such that $\Gamma(W)$ has an even labeled edge,  up to twist equivalence.

\begin{table}[ht]
\label{tab:nichols}
\begin{center}
\begin{tabular}{|c|c|c|c|}
\hline Type of $W$ & ${\mathcal B}(T_1,q^{\pm})$ & ${\mathcal B}(T_2,q^{\pm})$ & ${\mathcal B}(T,q^{\pm})$  \\
\hline  
$B_l,\,l\geq 3$& ${\mathcal B}(T_{D_l},q^{\pm})$& $(\bigwedge k)^{\otimes l}$,\, $\bigwedge \Cc^l$ & infinite dimensional\\
\hline
$F_4$&${\mathcal B}(T_{D_4},q^{\pm})$& ${\mathcal B}(T_{D_4},q^{\pm})$ & infinite dimensional\\
\hline
$I_2(4)$& $(\bigwedge \Cc)^{\otimes 2}$& $(\bigwedge \Cc)^{\otimes 2}$ & dimension $64$ \\
\hline
$I_2(2m),\,m>2$& ${\mathcal B}(T_{I_2(m)},q^{\pm})$& ${\mathcal B}(T_{I_2(m)},q^{\pm})$&infinite dimensional\\
\hline
\end{tabular}
\end{center}
\end{table}

We now focus on the irreducible, finite Coxeter groups with one conjugacy class of reflections.  

\begin{prop}\label{prop:dihedral-nichols}
 Let $W$ be of type $I_2(2m+1)$ for $m>1$. Then $\dim{\mathcal B}(T,q^{\pm})=\infty$. \end{prop}
\begin{proof}
Let $2m+1=\prod_{i=1}^rp_i^{e_i}$ be the prime factorization and let $W=\langle s,s'\rangle$.    
Then \[T=\{s(s's)^j~|~j=0,\,\ldots, 2m+1\}\]  and 
for any $n$ dividing $2m+1$, the subset $T_n=\{s(s's)^j\in T~|~n \textrm{ divides } j\}$ is a subrack of $T$ because $s(s's)^j\conj s(s's)^l=s(s's)^{2j-l}$, cf. \cite{AF}*{Remark 3.3}. In particular, $T$ contains a subrack of size $p_i$ for any prime divisor of $2m+1$. The pair $(T_{p_i}, q^+|_{T_{p_i}})$ is precisely the pair corresponding to the rack of reflections of the dihedral group of size $2p_i$ and $T_{p_i}$ is an indecomposable affine rack, \cite{AG}*{Section 1.3.8}.
If $p_i>7$ for some $i$, then \cite{HMV}*{Theorem 1.6} implies that $\dim{\mathcal B}(T_{p_i},q^+|_{T_{p_i}})=\infty$ and, a fortiori, $\dim{\mathcal B}(T,q^{\pm})=\infty$. Assume now that $p_i=5$ or $7$ for some $i$. In the notation of \cite{HMV}, the rack $T_{p_i}$ is ${\rm Aff}(p_i,p_i-1)$. Therefore, if $p_i=5,7$, the rack $T_{p_i}$ does not occur in \cite{HMV}*{Table 1} and so 
$\dim{\mathcal B}(T,q^{\pm})=\dim{\mathcal B}(T_{p_i},q^{\pm})=\infty$. We are left with the case $2m+1=3^b$ for some $b\geq 2$. In this case $T$ is an indecomposable affine rack, and $\dim{\mathcal B}(T_{3^b},q^{\pm})=\infty$ by \cite{AHV}*{Theorem 1.3}.
\end{proof}

\begin{cor}
 If $W$ is of type $H_3$ or $H_4$, then $\dim{\mathcal B}(T,q^{\pm})=\infty$.
\end{cor}
\bproof
The rack $T$ contains a subrack isomorphic to $T_{I_2(5)}$, so the statement follows from Theorem \ref{TheoremGabriel}, Remark \ref{rem:deco} \ref{item:inclusions} and Proposition \ref{prop:dihedral-nichols}.
\eproof

In the remaining cases  $W$ is in one of the crystallographic, simply-laced families of groups $A_n,\, n\geq 1$, $D_n$ for $n\geq 4$, and $E_6, E_7,E_8$.  They afford a crystallographic root system $\widetilde{\Phi}$ as in \cite{Bourbaki}*{Chapitre VI.1.1}, and all roots in $\widetilde{\Phi}$ have the same length. Here $T=\{s_\alpha~|~\alpha\in\widetilde{\Phi}^+\}$ and by \cite{Bourbaki}*{Chapitre VI.1.3} the subgroup generated by any pair of non-commuting reflections is isomorphic to ${\mathbb S}_3$. It is well-known that $\dim{\mathcal B}(T_{A_n},q^{\pm})<\infty$ for $n\leq 4$, \cites{MS,GGI}, whilst infinite-dimensionality for $n\geq 5$ is still open.  Observe that none of the splitting criteria in \cite{ACG}*{Sections 2.1, 2.2} and \cite{AFGV1}*{Section 3.2} apply to the rack $T$: in the terminology of \cite{ACG}, the rack $T$ is kthulhu. 

\begin{rem}{\rm It was kindly pointed out to us by I. Heckenberger that ${\mathcal B}(T_{D_4},q^{\pm})$ is infinite-dimensional as a consequence of \cite{AHV}*{Theorem 6.14} because the  Coxeter group of type $D_4$ is solvable, non-cyclic, and generated by $T$, which has size $>7$. Hence, $\dim{\mathcal B}(T_{D_n},q^{\pm})=\infty$ for any $n\geq4$.}
\end{rem}

In all remaining cases $\Gamma(W)$ contains a graph of type $A_5$ (i.e., ${\mathbb S}_6\leq W$). The case of ${\mathbb S}_6$ has been addressed by several authors: by the main result in \cite{Ba} it would be infinite-dimensional provided $\mathcal B(T,-1)$ is quadratic, but the latter property has not been established despite several attempts. Summarizing we have
\begin{cor}
Assume that $\dim{\mathcal B}(T_{A_5},q^{\pm})=\infty$. If $\dim{\mathcal B}(T,q^{\pm})<\infty$  then $W$ is either the dihedral group of order $8$, the cyclic group of order $2$, $\mathbb{S}_3$, $\mathbb{S}_4$ or $\mathbb{S}_5$.
\end{cor}

Through the equivalence of categories in \cite{KS}, the Nichols algebra ${\mathcal B}(T,q^{-})$ corresponds to a factorizable perverse sheaf whose underlying perverse sheaf is the intermediate extension of the local system on the ind-variety of configuration spaces in $\Cc$, corresponding to the collection of braid group representations $V^{\otimes n}$, for $n\geq 0$ associated with $T$ and the trivial cocycle. The latter are obtained as the push-forward of the constant sheaf on the Hurwitz space with Galois group $W$ and local monodromy $T$. One may hope to retrieve further information on these algebras (e.g. if they are quadratic) by using this geometric interpretation.

\subsection{Nichols algebras over dihedral groups}
In this section we restrict to the case in which $W$ is of rank $2$, that is, $W$ is of type $I_2(n)$ for $n>2$ and we write $S=\{s,s'\}$. 
The analysis of the finite-dimensional Nichols algebras over $I_{2}(n)$ when $4|n$ and $n\geq12$ was obtained in \cite{FG}, the cases $n=4,8$ were then completed in  \cite{BC}*{Chapter 2}. Thus in this section $n$ is not divisible by $4$ and we separate the analysis according to its parity. 

\begin{cor}\label{cor:odd}
Assume $n=2m+1$ for $m>1$. Then, any Nichols algebra of a Yetter-Drinfeld module of $W$ is infinite-dimensional. Therefore, if $H$ is a finite-dimensional complex, pointed Hopf algebra with group of grouplikes isomorphic to $W$, then $H={\mathbb C}W$, the group algebra of $W$. 
\end{cor}
\begin{proof}By \cite{AHS}*{Theorem 4.8}, the only possible finite-dimensional Nichols algebra coming from a Yetter-Drinfeld module for $W$ could come from the irreducible Yetter-Drinfeld module as in Remark \ref{rem:cohomologous} point \ref{item:YD} with $g\in S$  and $\eta$ the non-trivial irreducible representation of $H=\langle g\rangle$. However, this corresponds precisely to the pair $(T_{I_2(2m+1)}, q^+)$, cf. \cite{MS}*{Section 5}, which has an infinite-dimensional Nichols algebra by Proposition \ref{prop:dihedral-nichols}. The second statement follows from \cite{AG}*{\S 0.3}.
\end{proof}

We will now look at the case  $n=2r$, where $r$ is odd. Let $\zeta\in\Cc^*$ be a primitive $n$-th root of $1$ and let $C:=\langle ss'\rangle$. 
The following irreducible Yetter-Drinfeld modules over $W$ have a finite-dimensional associated Nichols algebra, \cite{AF}*{Theorem 3.1}. The action can be extracted from \cite{serre}*{5.3}. The grading of all of them is supported in $C$, thus to understand the braiding, it is enough to consider the $C$-action. This way, we can regard them as Yetter-Drinfeld modules over $C$.
\begin{itemize}
\item  $V_0:=\Cc v_0$, concentrated in degree $(ss')^r$ and with action $(ss')v_0=-v_0$. The Nichols algebra is $\bigwedge V_0$;
\item $V_{r,j}:=\Cc v_{+r,j}\oplus \Cc v_{-r,j}$, for $j\in\{1,\,\ldots,\,r-2\}$ and odd, with grading concentrated in degree $(ss')^r$ and action $(ss')v_{\pm r,j}=\zeta^{\pm j}v_{\pm r, j}$. The Nichols algebra is $\bigwedge V_{r,j}$;
\item $V_{h,j}:=\Cc v_{+h,j}\oplus \Cc v_{-h,j}$, for $j,h\in\{1,\,\ldots,\, r-2\}$, both odd, and such that $r|jh$, 
with $v_{\pm h,j}$ in degree $(ss')^{\pm h}$, respectively, and action $(ss')v_{\pm h,j}=\zeta^{\pm j}v_{\pm h,j}$. The Nichols algebra is $\bigwedge V_{h,j}$.
\end{itemize}

\begin{thm}\label{thm:2r}
Assume $n=2r$ for $r>3$ and odd. Let $V$ be a Yetter-Drinfeld module of $W$. Then $\dim{\mathcal B}(V)<\infty$ if and only if as a Yetter-Drinfeld module over $C$ there holds:
\begin{equation*}V\simeq V_0^{\oplus k}\oplus \bigoplus_{d=1}^{N} V_{h_d,j_d}^{\oplus k_d}\end{equation*} for some $k,k_d,N\geq0$, and some distinct pairs $(h_d,j_d)$ for $d\in\{1,\,\ldots,\,N\}$
with $h_d\in\{1,\,\ldots,\, r\}$ odd, $j_d\in\{1,\,\ldots,\,r-2\}$ odd and such that $r|(h_{d}j_{d'}+h_{d'}j_d)$ for all $d,d'\in\{1,\,\ldots,\,N\}$. In this case,
\begin{equation*}\label{eq:nichols}\mathcal{B}(V)\simeq \bigwedge V_0^{\oplus k}\;\widetilde{\otimes}\left(\widetilde\bigotimes_{d=1}^{N} \bigwedge V_{h_d,j_d}^{\oplus k_d}\right).\end{equation*} where $\widetilde{\otimes}$ stands for the tensor product of algebras twisted by the braiding of $V$.
\end{thm}
\bproof
 The analysis in \cite{AF}*{Theorem 3.1}, combined with Proposition \ref{prop:dihedral-nichols} and Remark \ref{rem:deco} point \ref{item:deco} shows that the only irreducible Yetter-Drinfeld modules of $W$ with finite-dimensional Nichols algebras are $V_0$ and $V_{h,j}$ with $h\in\{1,\,\ldots,\, r\}$ odd, $j\in\{1,\,\ldots,\,r-2\}$ odd and such that $r|jh$. Now we look at direct sums. 

A direct calculation shows that the braiding $c$ satisfies 
$c(x\otimes y)=-y\otimes x$ for $x\in V_0$, $y\in U$ or $x\in U$ and $y\in V_0$. Then we have the isomorphism  ${\mathcal B}(V_0^{\oplus k}\oplus U)\simeq \bigwedge V_0^{\oplus k}\,\widetilde{\otimes}{\mathcal B}(U)$ for any $k\geq 0$ after an iterated application of  Remark \ref{rem:cohomologous} point \ref{item:tens-prod}.

We now look at the braiding on $V_{h,j}\oplus V_{h',j'}$ where: $h,j,h',j'$ are odd; $1\leq h,h'\leq r$, $1\leq j,j'\leq r-2$, and sastisfy  $r|hj$ and $r|h'j'$. 
For $\epsilon,\epsilon'\in\{\pm 1\}$ we have \begin{equation}c(v_{\epsilon h,j}\otimes v_{\epsilon' h',j'})=\zeta^{\epsilon\epsilon' hj'}v_{\epsilon' h',j'}\otimes v_{\epsilon h,j}.\end{equation} 
Let  $\xi:=\zeta^{hj'+jh'}$.  If  $\xi=1$, that is, if $r|(hj'+jh')$, then $c^2|_{V_{h,j}\oplus V_{h',j'}}=\id$, and so ${\mathcal B}(V_{h,j}\oplus V_{h',j'})\simeq{\mathcal B}(V_{h,j})\widetilde{\otimes}{\mathcal B}(V_{h',j'})$. If, instead, $\xi\neq 1$, then we compute the generalized Dynkin diagram associated to the braiding on $V_{h,j}\oplus V_{h',j'}$ according to the recipe in \cite{He1}*{\S 2} obtaining  
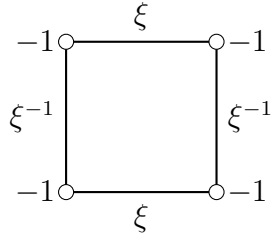
\begin{equation*}\label{eq:DD-braiding}
\begin{tikzpicture}
\draw[thick] (0.1,0)--(1.9,0) node[midway, below] {$\xi$};
\draw[thick] (0.1,2)--(1.9,2) node[midway,above] {$\xi$};
\draw[thick] (0,1.9)--(0,0.1) node[midway,left] {$\xi^{-1}$};
\draw[thick] (2,1.9)--(2,0.1) node[midway,right] {$\xi^{-1}$};
\draw (2,0) circle (1mm) node[right] {$-1$};
\draw (2,2) circle (1mm) node[right] {$-1$};
\draw (0,2) circle (1mm) node[left] {$-1$};
\draw (0,0) circle (1mm) node[left] {$-1$};
\end{tikzpicture}
\end{equation*}
The above diagram does not occur in \cite{He1}*{Table 3}, hence $\dim{\mathcal B}(V_{h,j}\oplus V_{h',j'})=\infty$ by \cite{He2}*{Theorem 3, Corollary 5}. 

Let now $U=\bigoplus_{d=1}^{N} V_{h_d,j_d}^{\oplus k_d}$ for some $k_d,N\geq0$, 
$h_d\in\{1,\,\ldots,\, r\}$ odd, $j_d\in\{1,\,\ldots,\,r-2\}$ odd and such that $r|h_dj_d$ for all $d$. 
If $r$ does not divide $h_{d}j_{d'}+h_{d'}j_d$ for some $d, d'$, then $\dim{\mathcal B}(U)=\infty$ by Remark \ref{rem:cohomologous} point \ref{item:inclusions}. If, instead,  $r|(h_{d}j_{d'}+h_{d'}j_d)$ for all $d, d'$, then the square of the braiding between any irreducible component of $U$ and the sum of the others is the identity. Thus, by Remark \ref{rem:cohomologous} point \ref{item:tens-prod} and induction on the number of irreducible components, we obtain 
\begin{equation*}{\mathcal B}(U)\simeq \widetilde{\bigotimes}_{d=1}^{N}{\mathcal B}(V_{h_d,j_d}^{\oplus k_d})\simeq \widetilde{\bigotimes}_{d=1}^{N}\left(\bigwedge V_{h_d,j_d}\right)^{\widetilde{\otimes}k_d}.\end{equation*}
Finally, observe that if $d=d'$ then $\zeta^{\epsilon\epsilon'h_dj_{d'}}=-1$, so $\left(\bigwedge V_{h_d,j_d}\right)^{\widetilde{\otimes}k_d}=\bigwedge V_{h_d,j_d}^{\oplus k_d}$. 
\eproof
We now complement  \cite{AF}*{Table 2} for $W=I_2(6)$ with the analysis of non-simple Yetter-Drinfeld modules. Here we have to take into account also the simple Yetter-Drinfeld modules supported in $T_1=W\conj s$ and $T_2=W\conj s'$, whose corresponding Nichols algebra is the Fomin-Kirillov algebra $FK_3$ of dimension $12$. 

The Yetter-Drinfeld modules whose associated braided spaces correspond to $(T_1, q^{\pm})$ are the modules $U_j$ for $j\in\{0,1\}$ with underlying vector space $\Cc e^j_{s}\oplus \Cc e^j_{s'ss'}\oplus \Cc e^j_{ss'ss's}$ where $e^j_g$ is in degree $g$ for $g\in W$, and the action is given by
\begin{align*}
&s\cdot e^j_s=-e^j_s,  &&s\cdot e^j_{s'ss'}=(-1)^{j+1}e^j_{ss'ss's},&& s\cdot e^j_{ss'ss's}=-(-1)^je^j_{s'ss'},\\
&(ss')\cdot e^j_s=e^j_{ss'ss's} ,&&(ss')\cdot e^j_{s'ss'}=(-1)^je^j_s,&&(ss')\cdot e^j_{ss'ss's}=e^j_{s'ss'}.
\end{align*}
  Similarly, the Yetter-Drinfeld modules  whose associated braided spaces correspond to $(T_2, q^{\pm})$ are the modules $U_j'$ for $j\in\{0,1\}$ with underlying vector space $\Cc e^j_{s'}\oplus \Cc e^j_{ss's}\oplus \Cc e^j_{s'ss'ss'}$ where $e_g$ is in degree $g$ for $g\in W$, and the action is given by
\begin{align*}
&s\cdot e^j_{s'}=-e^j_{ss's},  &&s\cdot e^j_{ss's}=-e^j_{s'},&& s\cdot e^j_{s'ss'ss'}=(-1)^{j+1}e^j_{s'ss's s'},\\
&(ss')\cdot e^j_{s'}=e^j_{ss's} ,&&(ss')\cdot e^j_{ss's}=e^j_{s'ss'ss'},&&(ss')\cdot e^j_{ss'ss's}=(-1)^je^j_{s'}.
\end{align*}
 In addition, to consider the braiding on sums of irreducible modules, we need to take into account the full action of $W$ on the modules $V_0$ and $V_{3,1}$ following \cite{AF}*{Theorem 3.1}. 
 According to \cite{serre}*{5.3} we have $s\cdot v_{\pm1,1}=v_{\mp 1,1}$ on $V_{3,1}$ and two possibilities for extending the action on $V_0$ from $C$ to $W$: we denote the two extensions by  $V_0^{j}:=\Cc v_0^{j}$ for $j\in\{0,1\}$  and set $s\cdot v_0^{j}=(-1)^j v_0^{j}$.

Twisting Yetter-Drinfeld modules as in Remark \ref{rem:cohomologous} point \ref{item:autom}, by the automorphism $\tau$ of $W$ that swaps $s$ and $s'$, preserves the isomorphism class of $V_{3,1}$ and interchanges $V_0^0$ and $V_0^1$. For $j=0,1$, it interchanges $U_j$ and $U_j'$.
 
\begin{prop}\label{prop:6}
Let $W=I_2(6)$ and $V$ be a Yetter-Drinfeld module for which $\dim{\mathcal B}(V)<\infty$. Then $V$ is isomorphic to one of the following modules: 
\begin{align*}
&V_{3,1}^{\oplus a}\oplus (V_0^0)^{\oplus b}\oplus (V_0^1)^{\oplus c},&&a,b,c\geq0, a+b+c\geq1;\\
&U_j\oplus(V_0^j)^{\oplus a}, \mbox{ or } U'_j\oplus(V_0^{1-j})^{\oplus a},&&j\in\{1,2\}, a\geq0;\\
&U\oplus V_{3,1},&&U\in\{U_0,U_1,U_0',U_1'\}
\end{align*}
and $\mathcal B(V)$ is, respectively,
\(\left(\bigwedge V_{3,1}\right)^{\widetilde\otimes a}\widetilde\otimes \left(\bigwedge V_0^0\right)^{\widetilde\otimes b}\otimes\left(\bigwedge V_0^1\right)^{\widetilde\otimes c};\) 
\(FK_3\widetilde\otimes \left(\bigwedge V_0^0\right)^{\widetilde\otimes a}\) and
the $2304$-dimensional Nichols algebra in \cite{HV}*{Theorem 8.2}, where $\widetilde\otimes$ denotes a twisted tensor product.
\end{prop}
\bproof By \cite{AF}*{Table 2}, if $V$ is irreducible then it is either $V_{3,1}$, $V_0^{j}$, $U_j$, or $U_j'$ for $j\in\{0,\,1\}$.
Sums of copies of $V_0^{0}$, $V_0^1$ and of $V_{3,1}$ can be handled as in the proof of Theorem \ref{thm:2r} because their grading is supported in $C$. Observe that the supports of $U_0$ and $U_1$ generate a group isomorphic to ${\mathbb S}_3$. Thus, by \cite{AHS}*{Theorem 4.8} if $\dim{\mathcal B}(U_0^a\oplus U_1^b)<\infty$ for some $(a,b)\neq(0,0)$, then $a+b=1$. The same argument applies to $U'_0$ and $U'_1$. 
For $\sigma\in T_1$ and $j,j'\in\{0,1\}$ we have $c^2(e_\sigma^j\otimes v_0^{j'})=(-1)^{j+j'}e_\sigma^j\otimes v_0^{j'}$ so ${\mathcal B}((V_0^{j'})^{\oplus k}\oplus U_j)$ is a twisted tensor product of ${\mathcal B}((V_0^{j'})^{\oplus k})$ and ${\mathcal B}(U_j)$ if and only if $j=j'$. Remark \ref{rem:cohomologous} point \ref{item:autom} implies that the same property holds for the pair $V_0^{1-j}=(V_0^j)^\tau$ and $U'_j=U_j^\tau$. 

For all other pairs $(X,Y)$ of irreducible modules the support of $X\oplus Y$ generates $W$ and $(c^2-\id)(X\oplus Y)\neq 0$. Under these conditions, \cite{HV-JA}*{Corollary 7.2} gives a precise list of the possible supports for $X\oplus Y$ such that  $\dim\mathcal B(X\oplus Y)<\infty$ and for each support, further conditions on $W$. 
In particular, if the size of the support is $6$, then $W$ must be a quotient of the group $\Gamma_4$ generated by $\mathrm a$, $\mathrm b$, $\nu$, such that $\nu^4=1$, $\nu\mathrm a=\mathrm a\nu^{-1}$, $\nu\mathrm b=\mathrm b\nu^{-1}$, and $\mathrm b\mathrm a=\nu \mathrm a\mathrm b$. A direct calculation shows that this is not the case.  Hence, $\dim{\mathcal B}(X\oplus Y)=\infty$ for $(X,Y)=(U_j,U'_{j'})$ for any $j,j'\in \{0,1\}$. We are left with the pairs $(U_0,V^1_0)$, $(U_1,V_0^0)$, $(U_0,V_{3,1})$, and $(U_1,V_{3,1})$ and their twistings by $\tau$. As a rack, their support is isomorphic to the one denoted by $Z^{3,1}_3$ in loc. cit. By \cite{HV}*{Theorem 2.1}, 
the only pair of Yetter-Drinfeld modules $X,Y$ of dimension respectively $3$ and $1$ and such that $X\oplus Y$ is supported by $Z^{3,1}_3$ is the one occurring in \cite{HV}*{Example 1.10}. There, it is required that, for $z$ in the support of $Y$ and $g$ in the support of $X$ there holds $1-\rho(z)\sigma(g)+(\rho(z)\sigma(g))^2=0$, where $\rho$ denotes the action on the homogeneous component $X_g$ and $\sigma$ denotes the action on the homogeneous component $Y_z$. This cannot apply to the pairs $(U_0,V^1_0)$ and $(U_1,V_0^0)$ because the supports contain only involutions. The pair $(U_0,V_{3,1})$ occurs in \cite{HV}*{Example 1.11} for $g=s$, $z=(ss')^3$ and $\epsilon=(ss')^2$, hence ${\mathcal B}(U_0\oplus V_{3,1})$ is the algebra described in \cite{HV}*{Theorem 8.4} whose dimension is $2304$. Now, as $(T_1,q^-)$ and $(T_1,q^+)$ are twist equivalent by Theorem \ref{TheoremGabriel} the cocycles corresponding to $U_0$ and $U_1$ are twist-equivalent, and therefore the cocycles corresponding to $U_0\oplus V_{3,1}$ is twist-equivalent to the cocycle corresponding to $U_1\oplus V_{3,1}$, so ${\mathcal B}(U_0\oplus V_{3,1})$ and ${\mathcal B}(U_1\oplus V_{3,1})$ are twist equivalent. We finally need to consider the Yetter-Drinfeld modules of the form $U_0\oplus V_{3,1}^{\oplus k}$ for $k>1$. They are all braid indecomposable because $(\id-c^2)(U_0\otimes V_{3,1})\neq0$, and $V_{3.1}$ is induced from a $2$-dimensional representation of $W$, that is the centraliser of $(ss')^3$. Hence this case is ruled out by \cite{HV-skeleton}*{Theorem 2.5}.    
\eproof

 Since the cases of $I_2(3)={\mathbb S}_3$ and $I_2(4m)$ are to be found in \cites{AHS, FG,BC},  Corollary \ref{cor:odd}, Theorem \ref{thm:2r} and Proposition \ref{prop:6} conclude the classification of finite-dimensional Nichols algebras over dihedral groups. 

\section{Acknowledgements}

We wish to thank N. Andruskiewitsch for useful discussions and I. Heckenberger for pointing out the content of Remark 4.5. 

This research has been carried out during an internship of Gabriel Maret at the Scuola Galileiana di Studi Superiori of the University of Padova, in the framework of the Programme Italie of the Ecole Normale Sup\'erieure de Paris. 

Project funded by the European Union – Next Generation EU under the National
Recovery and Resilience Plan (NRRP), Mission 4 Component 2 Investment 1.1 -
Call PRIN 2022 No. 104 of February 2, 2022 of Italian Ministry of
University and Research; Project 2022S8SSW2 (subject area: PE - Physical
Sciences and Engineering) ``Algebraic and geometric aspects of Lie theory''.


\begin{thebibliography}{20}
\bibitem{ACG}N. Andruskiewitsch, G. Carnovale, G. A. Garc\'ia, \emph{Finite-dimensional pointed Hopf algebras over finite simple groups of Lie type III. Semisimple classes in PSLn(q)}, Rev. Mat. Iberoam. 33 (2017), no. 3, 995--1024.

\bibitem{AF}N. Andruskiewitsch, F. Fantino,
\emph{On pointed Hopf algebras associated with alternating and dihedral groups}, Revista de la Uni\'on Matematica Argentina 48 (2007), no. 3,  57--71.

 \bibitem{AFGV}N. Andruskiewitsch, F. Fantino, G. A. Garc\'ia, L. Vendramin, \emph{On Nichols algebras associated to simple racks}, Contemp. Math 537 (2011), 31–56.

\bibitem{AFGaV}N. Andruskiewitsch, F. Fantino, G. A. Garc\'ia, L. Vendramin, \emph{On
twisted homogeneous racks of type D},  Revista de la Uni\'on Matem\'atica Argentina vol. 51-2 (2010), 1--16.

\bibitem{AFGV1}N. Andruskiewitsch, F. Fantino, M. Gra\~na, L. Vendramin, \emph{Finite-dimensional pointed Hopf algebras with alternating groups are trivial},  Ann. Mat. Pura Appl. (4) 190(2) (2011), 225--245 .

\bibitem{AG1}N. Andruskiewitsch and M. Gra\~na. \emph{Braided Hopf algebras over non-abelian finite groups}, Bol. Acad. Nacional de Ciencias (C\'ordoba) 63, 45--78 (1999). 

\bibitem{AG} N. Andruskiewitsch, M. Gra\~na, \emph{From racks to pointed Hopf algebras.}  Adv. Math. 178,  (2003)  no. 2, 177--243.

\bibitem{AHV}N. Andruskiewitsch, I. Heckenberger, L. Vendramin, Pointed Hopf algebras of odd dimension and Nichols algebras over solvable groups, Arxiv:2411.02304v1, (2024).

\bibitem{AHS}N. Andruskiewitsch, I. Heckenberger, H.-J. Schneider,
\emph{The Nichols algebra of a semisimple Yetter-Drinfeld module}, Amer.
J. Math. 132 6 (2010), 1493--1547.

\bibitem{AZ}N. Andruskiewitsch, S. Zhang, \emph{On pointed Hopf algebras associated to some conjugacy classes in $S_n$},  Proc. Amer. Math. Soc. 135  (2007), 2723--2731.

\bibitem{Ba}C. B\"arlingea, On the dimension of the Fomin-Kirillov algebra and related algebras,  arXiv:2001.04597v2, 2024.

\bibitem{Bazlov}
Y. Bazlov, \emph{Nichols-Woronowicz algebra model for Schubert calculus on Coxeter groups.} J. Algebra, 297(2):372–399, 2006.

\bibitem{BC} S. Beltr\'an Cubillos, \'Algebras de Nichols sobre grupos diedrales y pecios kthulhu en grupos espor\'adicos, PhD thesis, Universidad Nacional de C\'ordoba, 2020. 

\bibitem{BjornerBrenti}
A. Björner, F. Brenti, \emph{Combinatorics of Coxeter groups}, Graduate Texts in Mathematics 231, Springer, 2005.

\bibitem{Bourbaki}
N. Bourbaki, \emph{Groupes et algèbres de Lie}, Chapitre 4,5 et 6, \'Edition Masson, 1981.

\bibitem{CKS} J. S. Carter, S. Kamada, M. Saito \emph{Diagrammatic computations for quandles and cocycle knot invariants}, In: Diagrammatic Morphisms and Applications,
Proceedings of the AMS Special Session on Diagrammatic Morphisms in Algebra, Category Theory, and Topology, October, 21-22, 2000, San Francisco State University, Contemporary Math. 318, 
Editors: D. E. Radford, F. J. O . Souza  and D. N. Yetter.

\bibitem{Dold} C. Dold, \emph{Twist-equivalence of quadratic algebras associated to Coxeter groups}, 	PhD thesis, Manchester, 2016.

\bibitem{Dyer}M. Dyer, \emph{Reflection subgroups of Coxeter systems}, J. Algebra 206, (1998), 371--412.

\bibitem{ETV} J. S. Ellenberg, T.-T. Tran, C. Westerland, \emph{Fox-Neuwirth-Fuks cells, quantum shuffle algebras, and Malle’s conjecture for function fields}, ArXiv:1701.04541v2, 2023.

\bibitem{FG} F. Fantino, G.A. Garc\'ia, \emph{On pointed Hopf algebras over dihedral groups}, Pacific J. Math, Vol. 252 (2011), no. 1, 69--91

 \bibitem{FK} S. Fomin, A. Kirillov, \emph{Quadratic algebras, Dunkl elements, and Schubert calculus}, Advances in geometry, Progr. Math., vol. 172, Birkhäuser Boston, Boston, MA, 1999, 147--182.

 \bibitem{GGI} G. A. Garc\'ia, A. Garc\'ia Iglesias, \emph{Finite dimensional pointed Hopf algebras over S4}, Israel J. Math. 183 (2011), 417–444.

\bibitem{GeckPfeiffer}
M. Geck and G. Pfeiffer, \emph{Characters of Finite Coxeter Groups and Iwahori-Hecke Algebras}, London Math. Society Monographs, New Series, vol. 21, Oxford Science Publications, 2000.  


\bibitem{GranaJ} M. Gra\~na, \emph{A freeness theorem for Nichols algebras}, J. Algebra 231, 235--257, (2000).

\bibitem{He2}I. Heckenberger, \emph{Rank 2 Nichols Algebras with Finite Arithmetic Root System}, 
Algebr. Represent. Theory 11, 115--132,  (2008). 

\bibitem{He1} I. Heckenberger, \emph{Classification of arithmetic root systems}, 
Adv. Math. 220 59--124, (2009). 


\bibitem{HMV} I. Heckenberger, E. Meir, L. Vendramin, 
\emph{Finite-dimensional Nicholas algebras of simple Yetter-Drinfeld modules (over groups) of prime dimension}, Adv. Math. 444, 	Article 109637, (2024).

\bibitem{HS} I. Heckenberger, H.J. Schneider, \emph{Hopf algebras and root systems}, Mathematical Surveys and Monographs 247, AMS, 2020.

\bibitem{HV}I. Heckenberger, L. Vendramin, The classification of Nichols algebras over groups with finite root system of rank two, J. Eur. Math. Soc. 19, 1977--2017 (2017).

\bibitem{HV-JA}I. Heckenberger, L. Vendramin, Nichols algebras over groups with finite root system of rank two, III, Journal of Algebra 422, 223--256, (2015).

\bibitem{HV-skeleton}I. Heckenberger, L. Vendramin, A classification of Nichols algebras of semisimple
Yetter–Drinfeld modules over non-abelian groups, 
J. Eur. Math. Soc. 19, 299--356, (2017).  

\bibitem{H} R. B.  Howlett, \emph{On the Schur Multipliers of Coxeter Groups}, Journal of the London Mathematical Society Volume s2-38(2), 263--276, (1988).



\bibitem{KS}
{M.  Kapranov,  V. Schechtman},  
\emph{Shuffle algebras and perverse sheaves},  Pure and Applied Mathematics Quarterly,  Special Issue in Honor of Prof. Kyoji Saito' s 75th Birthday vol 16,  573--657 (2020), International Press.

 \bibitem{MO}S. Majid, R. Oeckl, \emph{Twisting of quantum differentials and the Planck scale Hopf algebra}, Comm. Math. Phys. 205 (1999), no. 3, 617–655.

 \bibitem{Michel}G. Michel, \emph{The Cohomology of Dihedral Nichols algebras using Koszul Complexes}, PhD thesis, University of Minnesota, 2020.

 \bibitem{MS}A. Milinski, H.-J. Schneider, \emph{Pointed indecomposable Hopf algebras over Coxeter groups}, In: New trends in Hopf algebra theory (La Falda, 1999), Contemp. Math., vol. 267, Amer. Math. Soc., Providence, RI, 2000,  215–-236.

\bibitem{serre}J. P. Serre, R\'epresentations Lin\'eaires des Groupes Finis, Hermann, Paris, 1971. 



\bibitem{Stembridge}
J. R. Stembridge, \emph{Quasi-Minuscule Quotients and Reduced Words for Reflections}, Journal of Algebraic Combinatorics 13 (2001), 275–-293.

\bibitem{Vendramin}
L. Vendramin, \emph{Nichols algebras associated to the transpositions of the symmetric group are twist-equivalent}, Proc. Amer. Math. Soc. 140 (2012), 3715--3723.


\end{thebibliography}
\end{document}